\documentclass[11pt]{amsart}
\usepackage[latin1]{inputenc}
\usepackage[english]{babel}
\usepackage{latexsym}
\usepackage{amsmath}
\usepackage{amssymb}
\usepackage{bbm}
\usepackage{eucal}
\usepackage{mathrsfs}
\usepackage{enumitem}
\usepackage{array}
\usepackage{ifthen}
\usepackage{graphicx}
\usepackage{hyperref}   
\usepackage{epstopdf}
\usepackage{todonotes}
\usepackage{marginnote}
\usepackage[all]{hypcap}
\hypersetup{
    colorlinks=true,       
    linkcolor=blue,        
    citecolor=blue,        
}

\newcommand{\set}[2]{\left\{#1: \, #2\right\}}
\newcommand{\norm}[2][]{\left\|#2\right\|}   
\newcommand{\sclp}[2][]{\langle#2\rangle} 

\newcommand{\N}{\mathbb{N}}

\newcommand{\R}{\mathbb{R}}
\newcommand{\C}{\mathbb{C}}

\newcommand{\eps}{\varepsilon}
\newcommand{\ii}{\mathrm{i}}
\newcommand{\e}{\mathrm{e}}

\newcommand{\dom}{D} 
\newcommand{\tr}{\mathrm{tr}} 
\newcommand{\dd}{\, \mathrm{d}}

\newcounter{satz}[section]
\renewcommand{\thesatz}{\arabic{section}.\arabic{satz}}

\newenvironment{theorem}[1][]
{\bigskip\noindent\refstepcounter{satz}\textsc{Theorem~\thesatz}\ifthenelse{\equal{#1}{}}{.}{~(\textit{#1}):} \, \it}
{\bigskip}

\newenvironment{proposition}[1][]
{\bigskip\noindent\refstepcounter{satz}\textsc{Proposition~\thesatz}\ifthenelse{\equal{#1}{}}{.}{~(\textit{#1}):} \, \it}
{\bigskip}

\newenvironment{corollary}[1][]
{\bigskip\noindent\refstepcounter{satz}\textsc{Corollary~\thesatz}\ifthenelse{\equal{#1}{}}{.}{~(\textit{#1}):} \, \it}
{\bigskip}

\newenvironment{remark}
{\bigskip\noindent\refstepcounter{satz}\textsc{Remark~\thesatz}.\,}
{\bigskip}

\begin{document}
\title[Weyl asymptotics]{Weyl type asymptotics and bounds for the eigenvalues of functional-difference operators for mirror curves}
\author{Ari Laptev}
\address{Ari Laptev, Department of Mathematics\\
Imperial College London, SW7 2AZ, London, UK \\
Institut Mittag-Leffler, Djursholm, Sweden}
\email{a.laptev@imperial.ac.uk}

\author{Lukas Schimmer}
\address{Lukas Schimmer, Department of Physics, Princeton University, Princeton, NJ 08544, USA}
\email{lukass@princeton.edu}

\author{Leon A. Takhtajan}
\address{Leon A. Takhtajan, Department of Mathematics \\
Stony Brook University\\ Stony Brook, NY 11794-3651 \\ USA; Euler Mathematical Institute, Saint Petersburg, Russia}
\email{leontak@math.stonybrook.edu}
\date{}

\begin{abstract} We investigate  Weyl type asymptotics of functional-difference operators associated to mirror curves of special del Pezzo Calabi-Yau threefolds. These operators are 
$H(\zeta)=U+U^{-1}+V+\zeta V^{-1}$ and $H_{m,n}=U+V+q^{-mn}U^{-m}V^{-n}$, where $U$ and $V$ are self-adjoint Weyl operators satisfying $UV=q^{2}VU$ with $q=\e^{\ii\pi b^{2}}$, $b>0$ and $\zeta>0$, $m,n\in\N$. We prove that $H(\zeta)$ and  $H_{m,n}$ are self-adjoint operators with purely discrete spectrum on $L^{2}(\R)$. Using the coherent state transform we find the asymptotical behaviour for the Riesz mean $\sum_{j\ge 1}(\lambda-\lambda_{j})_{+}$ as $\lambda\to\infty$ and prove the Weyl law for the eigenvalue counting function $N(\lambda)$ for these operators, which imply that their inverses are of trace class.
\end{abstract}

\maketitle

\section{Introduction}\label{sec:intro}
Let $P$ and $Q$ be quantum-mechanical momentum and position operators on $L^{2}(\R)$, satisfying on their common domain the Heisenberg commutation relation $[P,Q]=\ii I$. Consider the corresponding Weyl operators $U=\e^{-b P}$ and $V=\e^{2\pi b Q}$, where $b>0$. The operators $U$ and $V$ are 
unbounded self-adjoint operators on $L^{2}(\R)$, satisfying on their common domain the Weyl relation 
$$UV=q^{2}VU,$$
where $q=\e^{\ii\pi b^{2}}$.
In the coordinate representation $(P\psi)(x)=\ii\psi'(x)$ and $(Q\psi)(x)=x\psi(x)$, and the Weyl operators have the form $(U\psi)(x)=\psi(x+\ii b)$ and $(V\psi)(x)=\e^{2\pi b x}\psi(x)$. Their respective domains are
\begin{align*}
\dom(U)&=\set{\psi\in L^2(\R)}{\e^{-2\pi b k}\widehat{\psi}(k)\in L^2(\R)},\\
\dom(V)&=\set{\psi\in L^2(\R)}{\e^{2\pi b x}\psi(x)\in L^2(\R)},
\end{align*}
where $\mathcal{F}$ is the Fourier transform 
\begin{align*}
\widehat{\psi}(k)=(\mathcal{F}\psi)(k)=\int_\R\e^{-2\pi\ii kx}\psi(x)\dd x
\end{align*}
on $L^2(\R)$. Equivalently, $\dom(U)$ consists of those functions $\psi(x)$ which admit an analytic continuation to the strip $\set{z = x+\ii y\in\C}{0<y <b}$ such that $\psi(x +\ii y) \in L^2(\R)$ for all $0\leq y < b$ and there is a limit $\psi(x +\ii b - \ii 0) = \lim_{\eps\to 0^+}\psi (x + \ii b - \ii\eps)$ in the sense of convergence in $L^2(\R)$, which we will denote simply by $\psi(x+\ii b)$. The domain of $U^{-1}$ can be characterized similarly. 

Using the Weyl operators $U$ and $V$, one constructs the operator 
$$H=U+U^{-1}+V,$$
which in the coordinate representation becomes a functional-difference operator 
$$(H\psi)(x)=\psi(x+\ii b) + \psi(x-\ii b) + \e^{2\pi b x}\psi(x).$$
The operator $H$ first appeared in the study of the quantum Liouville model on the lattice \cite{Faddeev1986} and plays an important role in the representation theory of the non-compact quantum group $\mathrm{SL}_{q}(2,\R)$. In the momentum representation it becomes the Dehn twist operator in quantum Teichm\"{u}ller theory \cite{Kashaev2001}. In particular, in \cite{Kashaev2001} the eigenfunction expansion theorem for $H$ in the momentum representation was stated as formal completeness and orthogonality relations in the sense of distributions.

The spectral analysis of the functional-difference operator $H$ was done in \cite{Faddeev2014}. The operator $H$ was shown to be self-adjoint with a simple absolutely continuous spectrum $[2,\infty)$, and the eigenfunction expansion theorem for $H$, generalizing the classical Kontorovich-Lebedev transform, was proved. 

It was discovered in \cite{ADKMV} that the functional-difference operators built from the Weyl operators $U$ and $V$, also appear in the study of local mirror symmetry as a  quantization of an algebraic curve, the mirror to a toric Calabi-Yau threefold. The spectral properties of these operators were considered in \cite{GHM}. The typical example is a so-called local del Pezzo Calabi-Yau threefold, a total space of the anti-canonical bundle on a toric del Pezzo surface $S$. In the simplest case of the Hirzebruch surface $S=\mathbb{P}^{1}\times\mathbb{P}^{1}$ one gets the following operator
\begin{align} \label{eq:type1} 
H(\zeta)=\e^{-bP}+\e^{bP}+\e^{2\pi bQ}+\zeta\e^{-2\pi b Q}=U+U^{-1}+V+\zeta V^{-1}\,,
\end{align}
where $\zeta>0$ is a ``mass'' parameter, so that $H=H(0)$. In case $S$ is a weighted projective space $\mathbb{P}(1,m,n)$,  $m,n\in\N$, the corresponding operator is 
\begin{align} \label{eq:type2} 
H_{m,n}=\e^{-bP}+\e^{2\pi b Q}+\e^{bmP-2\pi bnQ}=U+V+q^{-mn}U^{-m}V^{-n},
\end{align}
and $H=H_{1,0}$ (see \cite{GHM} for details). It was conjectured in \cite{GHM} for the cases $\zeta>0$ and $m,n\in\N$ that these operators have a discrete spectrum, their inverses are of trace class and their Fredholm determinants can be explicitly evaluated in terms of enumerative invariants of the underlying Calabi-Yau threefolds. In a recent paper \cite{Kashaev2015} some of these conjectures were proved and the authors obtained a remarkable explicit formula for the operators $H(\zeta)^{-1}$ and $H^{-1}_{m,n}$ in terms of the modular quantum dilogarithm.

The present paper is devoted to the study of  Weyl type asymptotics for the operators $H(\zeta)$ and $H_{m,n}$ as self-adjoint operators on $L^{2}(\R)$.
Namely, we prove that they are operators with purely discrete spectrum and investigate the asymptotic behavior of their eigenvalues, from which it immediately follows that $H(\zeta)^{-1}$ and $H_{m,n}^{-1}$ are of trace class. Our main results are Theorems \ref{th:Hasymp} and \ref{th:Gasymp} on the asymptotic behaviour of the Riesz mean $\sum_{j\ge 1}(\lambda-\lambda_{j})_{+}$ and Corollaries \ref{cor:Hasymp} and \ref{cor:Gasymp} on the Weyl law for the eigenvalue counting function $N(\lambda)$ for these operators. Namely,
\begin{equation} \label{W-zeta}
\lim_{\lambda\rightarrow\infty}\frac{N(\lambda)}{\log^{2}\lambda}=\frac{1}{(\pi b)^{2}}
\end{equation}
for the operator $H(\zeta)$ and
\begin{equation} \label{W-mn}
\lim_{\lambda\rightarrow\infty}\frac{N(\lambda)}{\log^{2}\lambda}=\frac{c_{m,n}}{(2\pi b)^{2}},\quad c_{m,n}=\frac{(m+n+1)^{2}}{2mn}
\end{equation}
for the operator $H_{m,n}$. The proof follows ideas developed in \cite{Laptev1997}, where the Fourier transform is replaced by the coherent state transform. The applied methods also mimic the derivation of the Berezin--Lieb inequality \cite{Berezin1972,Berezin1972b,Lieb1973}.

\subsection*{Acknowledgements} The work of L.S.~  was supported by the NSF grant PHY-1265118 at Princeton University. The work of L.T. was partially supported by the NSF grant DMS-1005769. A.L. is grateful to T. Weidl for useful discussions. Some results of the paper were obtained by using different methods by A.L.'s master student O. Mickelin \cite{Mickelin2015}. The research was supported by the Russian Foundation for Basic Research (grant no. 12-01-00203).

\section{The Operator $H(\zeta)$}\label{sec:H}

Let $H_0=U+U^{-1}$ and $W(\zeta)=V+\zeta V^{-1}$ so that $H(\zeta)=H_0+W(\zeta)$ and formally
\begin{align*}
\big(H(\zeta)\psi\big)(x)=\psi(x+\ii b)+\psi(x-\ii b)+(\e^{2\pi b x}+\zeta\e^{-2\pi bx})\psi(x)\,.
\end{align*}
It is straightforward to show that $\mathcal{F}H_0\mathcal{F}^{-1}=W$, where we put $W=W(1)$, which yields $\sigma(H_0)=[2,\infty)$ and consequently $H\ge2I$. The operator $H(\zeta)$ is semi-bounded and symmetric on the common domain of $H_{0}$ and $W(\zeta)$, 
$$\langle H(\zeta)\psi,\psi\rangle=\langle\psi,H(\zeta)\psi\rangle\ge 2\Vert\psi\Vert^{2},$$
where $\langle~,~\rangle$ stands for the inner product in $L^{2}(\R)$. Thus 
we can define a self-adjoint Friedrichs extension of the operator $H(\zeta)$ (see e.g. \cite[Chapter 10.3]{Birman1987}). It is this extension that we mean when we refer to the operator $H(\zeta)$. 
We first show that the spectrum of $H(\zeta)$ is purely discrete. 

\begin{proposition}\label{prop:Hspec}
Let $L(x)$ be a continuous, real-valued, bounded below function such that $L(x)$ tends to $+\infty$ as $|x|\rightarrow\infty$.
Then the operator $T=H_{0}+L$ has purely discrete spectrum consisting of finite multiplicity eigenvalues tending to $+\infty$.
\end{proposition}
\begin{proof}
\noindent
Indeed, by using the variational principle and the Birman--Schwinger principle we have
\begin{align*} 
&\dim\set{\psi}{\sclp{T\psi,\psi} < \lambda \sclp{\psi,\psi}}\\
&\le \dim\set{\psi}{\sclp{H_{0 }\psi, \psi} - \left(\lambda\norm{\psi}^2 - \sclp{L\psi,\psi}\right)_+ <0}\\
&= 
\dim\set{\psi}{\sclp{W_\lambda H_{0}^{-1} W_\lambda \psi,\psi} > 1},
\end{align*} 
where $W_\lambda= \sqrt{(\lambda-L)_+}$. 
The operator $K_\lambda =W_\lambda H_{0}^{-1} W_\lambda$ is an integral operator
\begin{align*}
(K_\lambda\psi)(x)= \int_{-\infty}^\infty\int_{-\infty}^\infty W_\lambda(x)\frac{\e^{2\pi \ii(x-y)k}}{2\cosh( 2\pi b k)} W_\lambda(y) \psi(y) \dd k \dd y.
\end{align*}
Since $L(x)$ tends to $+\infty$ as $|x|\to \infty$, the support of  $W_\lambda$ is compact. 
Therefore $K_\lambda$ is a compact operator and this proves that its spectrum above one is finite. This implies that the spectrum of $T$ below $\lambda$ is also finite for any fixed $\lambda>0$.

Clearly $T$ cannot have finite rank since it is the sum of two unbounded positive operators. Therefore the spectrum of the operator $T$ is discrete.
\end{proof}

Let $\lambda_1\le\lambda_2\le\dots$ denote the eigenvalues of $H(\zeta)$ with the corresponding complete system of orthonormal eigenfunctions $\psi_j\in L^2(\R)$.
We are interested in the asymptotic behaviour of the Riesz mean $\sum_{j\ge1}(\lambda-\lambda_j)_+$ as $\lambda\to\infty$. Here $x_+=(|x|+x)/2$ is defined as the positive part of a real number $x$. Our main result is the following. 

\begin{theorem}\label{th:Hasymp}
For any $\zeta>0$ the eigenvalues $\lambda_j$ of the operator $H(\zeta)$ have the following asymptotic behaviour
\begin{align}
\sum_{j\ge1}(\lambda-\lambda_j)_+=\frac{\lambda\log^{2}\lambda}{(\pi b)^2}+O(\lambda\log\lambda)\quad\text{as}\quad\lambda\to\infty.
\label{eq:Hasymp} 
\end{align}
\end{theorem}

The following is an immediate consequence of Theorem \ref{th:Hasymp}.

\begin{corollary}\label{cor:Hasymp}
For any $b>0$ the number $N(\lambda)=\#\set{j\in\N}{\lambda_j<\lambda}$ of eigenvalues of $H(\zeta)$ less than $\lambda$ satisfies 
\begin{align*}
\lim_{\lambda\to\infty}\frac{N(\lambda)}{\log^{2}\lambda}=\frac{1}{(\pi b)^2}\,.
\end{align*}

\end{corollary}

In particular, the operator $H(\zeta)^{-1}$ is of trace class since
$$\sum_{j=1}^{\infty}\frac{1}{\lambda_{j}}=\int_{2}^{\infty}\frac{1}{\lambda}\dd N(\lambda)=\left. \frac{N(\lambda)}{\lambda}\right|_{2}^{\infty}+\int_{2}^{\infty}\frac{N(\lambda)}{\lambda^{2}}\dd \lambda<\infty.$$
\begin{remark}\label{Weyl law}
Theorem \ref{th:Hasymp} and Corollary \ref{cor:Hasymp} are Weyl type results that link the asymptotical behaviour of quantum mechanical expressions to classical phase space integrals.  
Namely, let
$$\sigma(k,x)=2\cosh(2\pi bk)+\e^{2\pi b x}+\zeta\e^{-2\pi b x}$$
be the total symbol of the operator $H(\zeta)$. Then the term $\log^{2}\lambda/(\pi b)^2$ is precisely the leading term of the phase volume of the classical region 
$\{(k,x)\in\R^{2}: \sigma(k,x)\leq\lambda\}$ as $\lambda\to\infty$. Similarly, $\lambda\log^{2}\lambda/(\pi b)^2$ coincides with the leading term in the phase space integral 
$$\iint_{\R^{2}}(\lambda-\sigma(k,x))_{+}\dd k\dd x\quad\text{as}\quad\lambda\to\infty.$$
\end{remark}

To prove Theorem \ref{th:Hasymp}, we establish lower and upper bounds on the Riesz mean $\sum_{j\ge1}(\lambda-\lambda_j)_+$ in Sect.~\ref{subsec:upper} and \ref{subsec:lower} respectively. To this end we introduce the coherent state representation of $H(\zeta)$. To simplify notation and to keep focus on the arguments involved, we concentrate on the case $\zeta=1$, where $W=2\cosh(2\pi b x)$. Subsequently, we will be using notation $H=H_{0}+W$, not to be confused with the operator $H_{0}+V$. The general case $\zeta>0$ is a straightforward generalization, as is explained in Sect.~\ref{subsec:Hgen}. 

\subsection{The Coherent State Representation}\label{subsec:coherent}
Let $g$ be the Gaussian function $g(x)=(a/\pi)^{1/4}\e^{-\frac a2 x^2}$ with some $a>0$. Clearly $g$ satisfies $\norm[2]{g}=1$ in $L^2(\R)$. 
For $\psi\in L^2(\R)$ the classical coherent state transform (see e.g. \cite[Chapter 12]{Lieb2001}) is given by
\begin{align*}
\widetilde{\psi}(k,y)=\int_\R\e^{-2\pi\ii k x}g(x-y)\psi(x)\dd x\,.
\end{align*}
Denoting by $(f*g)(x)=\int_\R f(x-y)g(y)\dd y$ the convolution of $f$ and $g$, Plancherel's theorem shows that
\begin{align}
\int_\R|\widetilde{\psi}(k,y)|^2\dd k&=(|\psi|^2*|g|^2)(y)\,,\label{eq:Itrans1}\\
\int_\R|\widetilde{\psi}(k,y)|^2\dd y&=(|\widehat{\psi}|^2*|\widehat{g}|^2)(k)\label{eq:Itrans2}\,.
\end{align}
The proof of the second identity also uses the convolution theorem. 

We aim to find representations of $\sclp[2]{H_0\psi,\psi}$ and $\sclp[2]{W\psi,\psi}$ in terms of coherent states. 
It follows from  \eqref{eq:Itrans2} that
\begin{align*}
\iint_{\R^{2}} 2\cosh(2\pi b k)|\widetilde{\psi}(k,y)|^2\dd k\dd y
&=\iint_{\R^{2}} 2\cosh(2\pi b k)|\widehat{\psi}(k-q)|^2|\widehat{g}(q)|^2\dd k\dd q,
\end{align*}
and using 
\begin{align*}
\cosh(x+y)=\cosh x\cosh y+\sinh x\sinh y
\end{align*} 
we obtain 
\begin{gather*}
\iint_{\R^{2}} 2\cosh(2\pi b k)|\widetilde{\psi}(k,y)|^2\dd k\dd y\\
=\iint_{\R^{2}} 2\cosh\big(2\pi b(k-q)\big)|\widehat{\psi}(k-q)|^2\cosh(2\pi b q)|\widehat{g}(q)|^2\dd k\dd q\\
+\iint_{\R^{2}} 2\sinh\big(2\pi b(k-q)\big)|\widehat{\psi}(k-q)|^2\sinh(2\pi bq)|\widehat{g}(q)|^2\dd k\dd q\,.
\end{gather*}
Recalling that $\mathcal{F}H_0\mathcal{F}^{-1}=W$, the first integral on the right-hand side can be computed to be $\frac12\sclp[2]{H_0\psi,\psi}\sclp[2]{W\widehat{g},\widehat{g}}$. Since $g(x)=g(-x)$, it holds that $\widehat{g}(k)=\widehat{g}(-k)$ and consequently the second integral vanishes. Thus for $\psi\in\dom(H_0)$ we obtain the representation 
\begin{align}
\sclp[2]{H_0\psi,\psi}=d_1\iint_{\R^{2}} 2\cosh(2\pi b k)|\widetilde{\psi}(k,y)|^2\dd k\dd y
\label{eq:rep1} 
\end{align}
where
$$d_1=\frac{2}{\sclp[2]{W\widehat{g},\widehat{g}}}=\e^{-ab^2/4}<1.$$

Similarly, we can use \eqref{eq:Itrans1} to compute that 
\begin{align*}
\iint_{\R^{2}} 2\cosh(2\pi by)|\widetilde{\psi}(k,y)|^2\dd k\dd y
&=\iint_{\R^{2}} 2\cosh(2\pi by)|{\psi}(y-q)|^2|{g}(q)|^2\dd y\dd q,
\end{align*}
which with the help of the same trigonometric identity as above can be simplified to 
\begin{align*}
\iint_{\R^{2}} 2\cosh(2\pi by)|\widetilde{\psi}(k,y)|^2\dd k\dd y &=\frac12\sclp[2]{W\psi,\psi}\sclp[2]{Wg,g}.
\end{align*}
Thus for $\psi\in\dom(W)$ we have the representation 
\begin{align}
\sclp[2]{W\psi,\psi}=d_2\iint_{\R^{2}} 2\cosh(2\pi b y)|\widetilde{\psi}(k,y)|^2\dd k\dd y,
\label{eq:rep2} 
\end{align}
where 
$$d_2=\frac{2}{\sclp[2]{Wg,g}}=\e^{-(\pi b)^2/a}<1.$$
Summarizing, we obtain
\begin{equation} \label{H-coherent}
\sclp[2]{H\psi,\psi}=\iint_{\R^{2}}2(d_1\cosh(2\pi b k)+d_2\cosh(2\pi by))|\widetilde{\psi}(k,y)|^2\dd k\dd y.
\end{equation}

\subsection{Deriving an Upper Bound} \label{subsec:upper}
We apply ideas that were used in \cite{Laptev1997} in investigation of the upper bounds on the eigenvalues of a general class of operators on sets of finite measure with Dirichlet boundary condition. While these results relied on the representation of the operators in Fourier space, we will use the representation in terms of the coherent states.  

As a reminder, $\lambda_j$ denote the eigenvalues of $H$ and $\psi_j$ the corresponding orthonormal eigenfunctions which form a complete set. 
We first observe that representation \eqref{H-coherent} yields
\begin{gather*}
\sum_{j\ge1}(\lambda-\lambda_j)_+
=\sum_{j\ge1}(\lambda-\sclp[2]{H\psi_j,\psi_j})_+\\
=\sum_{j\ge1}\left(\lambda- \iint_{\R^{2}} 2\big(d_1 \cosh(2\pi bk)+ d_2 \cosh(2\pi by)\big)|\widetilde{\psi}_j(k,y)|^2\dd k\dd y\right)_+.
\end{gather*}
By Plancherel's theorem it holds that
\begin{equation}
\iint_{\R^{2}}|\widetilde{\psi}_j(k,y)|^2\dd k\dd y=\norm[2]{\psi_j}^2=1
\label{eq:coherentnorm} 
\end{equation}
and consequently we can apply Jensen's inequality with the convex function $x\mapsto (\lambda-x)_+$ to obtain
\begin{gather*}
\sum_{j\ge1}(\lambda-\lambda_j)_+\\
\le \iint_{\R^{2}} \big(\lambda-2d_1\cosh(2\pi bk)- 2d_2\cosh(2\pi by)\big)_+\sum_{j\ge1}|\widetilde{\psi}_j(k,y)|^2\dd k\dd y\,.
\end{gather*}
Put $e_{k,y}(x)=\e^{2\pi\ii kx}g(x-y)$. Since the eigenfunctions $\psi_j$ form an orthonormal basis in $L^{2}(\mathbb{R})$, 
\begin{align*}
\sum_{j=1}^{\infty}|\widetilde{\psi}_j(k,y)|^2=\sum_{j=1}^{\infty}|(e_{k,y},\psi_{j})|^{2}=\Vert e_{k,y}\Vert^{2}=1\quad\text{for all}\quad k,y\in\R,
\end{align*}
and we arrive at the upper bound
\begin{align*}
\sum_{j\ge1}(\lambda-\lambda_j)_+\le \iint_{\R^{2}} \big(\lambda-2d_1\cosh(2\pi bk)- 2d_2\cosh(2\pi by)\big)_+\dd k\dd y\,.
\end{align*}

To investigate the behaviour of the integral on the right-hand side as $\lambda\to\infty$, we first note that \begin{align*}
\sum_{j\ge1}(\lambda-\lambda_j)_+
&\le4\int_0^\infty\int_0^\infty \big(\lambda-2d_1\cosh(2\pi bk)- 2d_2\cosh(2\pi by)\big)_+\dd k\dd y\\
&\le4\int_0^\infty\int_0^\infty \big(\lambda-d_1 \e^{2\pi bk}- d_2 \e^{2\pi by}\big)_+\dd k\dd y\,,
\end{align*}
where we used that $2\cosh x>\e^{x}$ for $x>0$.
Changing the variables $u_1=d_1\e^{2\pi bk}, u_2=d_2\e^{2\pi by}$ we arrive at
\begin{align*}
\sum_{j\ge1}(\lambda-\lambda_j)_+
&\le \frac{1}{(\pi b)^2}\int_{d_1}^\infty\int_{d_2}^\infty\frac{(\lambda-u_1-u_2)_+}{u_1u_2}\dd u_2\dd u_1\\
&=\frac{1}{(\pi b)^2}\int_{d_1}^{\lambda-d_2}\int_{d_2}^{\lambda-u_1}\frac{\lambda-u_1-u_2}{u_1u_2}\dd u_2\dd u_1\,,
\end{align*}
where $\lambda\ge d_1+d_2$ since $\lambda\ge2$ and $d_1,d_2\le 1$.  Now we immediately obtain
\begin{align*}
\int_{d_1}^{\lambda-d_2}\int_{d_2}^{\lambda-u_1}\frac{\lambda-u_1-u_2}{u_1u_2}\dd u_2\dd u_1&=\lambda\int_{d_1/\lambda}^{1-d_2/\lambda}\int_{d_2/\lambda}^{1-v_{1}}\frac{1-v_1-v_2}{v_1v_2}\dd v_2\dd v_1\\
&= \lambda\log^{2}\lambda+O(\lambda\log\lambda)
\end{align*}
as $\lambda\rightarrow\infty$, so that
\begin{align*}
\sum_{j\ge 1}(\lambda-\lambda_j)_+\le \frac{\lambda\log^{2}\lambda}{(\pi b)^2}+O(\lambda\log\lambda).
\end{align*}

\subsection{Deriving a Lower Bound}\label{subsec:lower}
To obtain a lower bound, we use a different argument. 
The ideas in this section are again taken from \cite{Laptev1997}, where a lower bound on the eigenvalues of a general class of operators on sets of finite measure with Neumann boundary condition was obtained. Similarly to the previous subsection, the coherent state transform will replace the Fourier transform.   

Recalling \eqref{eq:coherentnorm}, we start from the identity
\begin{align*}
\sum_{j\ge1}(\lambda-\lambda_j)_+
=\sum_{j\ge1}(\lambda-\lambda_j)_+\iint_{\R^{2}}|\widetilde{\psi}_j(k,y)|^2\dd k\dd y,
\end{align*}
and observing that 
\begin{equation*}
\widetilde{\psi}_j(k,y)=\int_\R\psi_j(x)\overline{e_{k,y}(x)}\dd x=\sclp[2]{\psi_j,e_{k,y}},
\end{equation*}
we obtain 
\begin{align*}
\sum_{j\ge1}(\lambda-\lambda_j)_+ 
&=\iint_{\R^{2}} \sum_{j\ge1}(\lambda-\lambda_j)_+\sclp[2]{\psi_j,e_{k,y}}\overline{\sclp[2]{\psi_j,e_{k,y}}}\dd k\dd y\\
&=\iint_{\R^{2}} \sum_{j\ge1}(\lambda-\lambda_j)_+\big\langle\sclp[2]{e_{k,y},\psi_j}\psi_j,e_{k,y}\big\rangle\dd k\dd y\,.
\end{align*}
Denoting by $\mathrm{d} E_{\mu}$ the projection-valued measure for $H$ on $[2,\infty)$, we conclude that
\begin{align*}
\sum_{j\ge1}(\lambda-\lambda_j)_+
=\iint_{\R^{2}}\int_2^\infty(\lambda-\mu)_+\sclp[2]{\mathrm{d} E_{\mu} e_{k,y},e_{k,y}}\dd k\dd y\,.
\end{align*}
Since by the spectral theorem
\begin{align*}
\int_2^\infty \sclp[2]{\mathrm{d} E_{\mu} e_{k,y},e_{k,y}}=\sclp[2]{e_{k,y},e_{k,y}}=\norm[2]{g}^2=1,
\end{align*}
we can apply Jensen's inequality with the convex function $x\mapsto(\lambda-x)_+$ and obtain the lower bound
\begin{align}
\sum_{j\ge1}(\lambda-\lambda_j)_+ 
\ge \iint_{\R^{2}}\left(\lambda-\int_2^\infty \mu\sclp[2]{\mathrm{d} E(\mu) e_{k,y},e_{k,y}}\right)_+\dd k\dd y\,.
\label{eq:lbound} 
\end{align}
Again it follows from the spectral theorem that
\begin{align*}
\int_2^\infty \mu\sclp[2]{\mathrm{d} E(\mu) e_{k,y},e_{k,y}}
=\sclp[2]{He_{k,y},e_{k,y}}=\sclp[2]{H_0e_{k,y},e_{k,y}}+\sclp[2]{We_{k,y},e_{k,y}}\,.
\end{align*}
 
The two terms on the right-hand side can be computed explicitly. We first consider $\sclp[2]{He_{k,y},e_{k,y}}$ and note that
\begin{align*}
g(x-y\pm\ii b)=\e^{\frac{ab^2}{2}}g(x-y)\e^{\mp a(x-y)\ii b},
\end{align*}
whence
\begin{align*}
\sclp[2]{H_0e_{k,y},e_{k,y}}
&=\int_\R\big(\e^{-2\pi bk}g(x-y+\ii b)+\e^{2\pi bk}g(x-y-\ii b)\big)g(x-y)\dd x\\
&=\e^{\frac{ab^2}{2}}\left(\e^{-2\pi bk}\int_\R g(z)^2\e^{-\ii abz}\dd z+\e^{2\pi bk}\int_\R g(z)^2\e^{\ii abz}\dd z\right)\\
&=\frac{1}{d_1}2\cosh(2\pi bk)\,.
\end{align*}
For the second term, $\sclp[2]{We_{k,y},e_{k,y}}$, we get 
\begin{align*}
\sclp[2]{We_{k,y},e_{k,y}}
&=\int_\R2\cosh(2\pi bx)g(x-y)^2\dd x\\
&=\int_\R2\cosh\big(2\pi b(x-y)\big)\cosh(2\pi by) g(x-y)^2\dd x
\\&\phantom{=}+\int_\R2\sinh\big(2\pi b(x-y)\big)\sinh(2\pi by)g(x-y)^2\dd x\\
&=\frac{1}{d_2}2\cosh(2\pi by)\,.
\end{align*}

Combining these two results with \eqref{eq:lbound} we arrive at
\begin{align*}
\sum_{j\ge1}(\lambda-\lambda_j)_+ 
&\ge\iint_{\R^{2}}\left(\lambda-\frac{2}{d_1}\cosh(2\pi bk)-\frac{2}{d_2}\cosh(2\pi by)\right)_+\!\dd k\dd y\\
&=4 \int_0^\infty\int_0^\infty\left(\lambda-\frac{2}{d_1}\cosh(2\pi bk)-\frac{2}{d_2}\cosh(2\pi by)\right)_+\!\dd k\dd y\,. 
\end{align*}
Note that $2\cosh x\le 2\e^{x}$ for $x\ge0$ and thus 
\begin{align*}
\sum_{j\ge1}(\lambda-\lambda_j)_+ 
\ge 4 \int_0^\infty\int_0^\infty\left(\lambda-\frac{2}{d_1}\e^{2\pi bk}-\frac{2}{d_2}\e^{2\pi by}\right)_+\dd k\dd y\,.
\end{align*}
The integral on the right-hand side is computed in the same way as in the previous section. The only difference is that the numbers $d_1, d_2$ have been replaced by $2/d_1,2/d_2$. These coefficients have no influence on the leading term for large $\lambda$ as long as $\lambda\ge2/d_1+2/d_2$, and we conclude 
\begin{align*}
\sum_{j\ge1}(\lambda-\lambda_j)_+ 
\ge  \frac{1}{(\pi b)^2}\lambda\log^{2}\lambda+O(\lambda\log\lambda)\quad\text{as}\quad\lambda\to\infty.
\end{align*}

\subsection{The Number of Eigenvalues}\label{subsec:N}
We present two proofs of Corollary \ref{cor:Hasymp}. One uses the Karamata--Tauberian theorem \cite{Karamata1931} to deduce it from Theorem \ref{th:Hasymp}, while the other 
consists in obtaining the optimal bounds for $N(\lambda)$ from the Riesz mean.
\subsubsection{Proof of Corollary \ref{cor:Hasymp} with the Karamata--Tauberian Theorem}
The Karamata--Tauberian theorem (for the proof see, e.g.,  \cite[Theorem 10.3]{Simon2005}) connects the asymptotic behaviour of $N(\lambda)$ for large $\lambda$ to the divergence of $\tr\,\e^{-tH}$ for small $t$. In \cite{Laptev1999} a general method was discussed that allows to obtain asymptotics of the traces of convex functions of self-adjoint operators from the behaviour of their Riesz means. Namely, from the representation 
\begin{align*}
\e^{-t\lambda}=t^2 \int_\R (s-\lambda)_+\e^{-ts}\dd s
\end{align*}  
for $\lambda\ge0$ and asymptotic behaviour \eqref{eq:Hasymp} we get the upper bound
\begin{align*}
\tr\,\e^{-tH}&=t^2 \int_0^\infty \sum_{j\ge1}(s-\lambda_j)_+\e^{-ts}\dd s\\ 
&\le\frac{t^2}{(\pi b)^2}\int_0^\infty s(\log s)^2\e^{-ts}\dd s +t^2C\int_0^\infty s(\log s)\e^{-ts}\dd s
\end{align*}
with some constant $C>0$, 
as well as a similar lower bound with a different constant. The two integrals on the right-hand side are computed explicitly and we obtain
\begin{align*}
\lim_{t\to0}\frac{\tr\,\e^{-tH}}{\log^{2} t}=\frac{1}{(\pi b)^2}\,.
\end{align*}
A slight modification of the Karamata--Tauberian theorem that allows for logarithmic terms \cite{Simon1983} implies that
\begin{align*}
\lim_{\lambda\to\infty}\frac{N(\lambda)}{\log^{2}\lambda}=\frac{1}{(\pi b)^2}\,.
\end{align*}

\subsubsection{Direct Proof of Corollary \ref{cor:Hasymp}}
To derive an upper bound on $N(\lambda)$, we let $\mu\ge\rho>0$ and note the that
$$\sum_{j\ge1}(\mu-\lambda_j)_+=\sum_{\lambda_{j}<\mu}(\mu-\lambda_{j})\ge\sum_{\lambda_{j}<\mu-\rho}(\mu-\lambda_{j})>\rho N(\mu-\rho).$$
We can now use  asymptotic behaviour \eqref{eq:Hasymp} of the Riesz mean to conclude that there exists a $C>0$ such that
\begin{align*}
N(\mu-\rho)\le\frac{\mu\log^{2}\mu}{\rho(\pi b)^2}+\frac{C}{\rho}\mu\log\mu\,.
\end{align*}
With $\tau>0$ we now choose $\mu=(1+\tau)\lambda$ and $\rho=\tau\lambda$ such that $\mu-\rho=\lambda$ and
\begin{align*}
N(\lambda)\le\frac{1}{(\pi b)^2}\left(1+\frac{1}{\tau}\right)\left(\log^{2}(\lambda+\lambda\tau)+C\log(\lambda+\lambda\tau)\right).
\end{align*}
It remains to optimize this upper bound with respect to $\tau>0$. The minimum is attained at $\tau_0$ defined by  the equation
\begin{align*}
2\tau_0=\log(\lambda+\lambda\tau_0)\,.
\end{align*}
Since $2\tau-\log(1+\tau)$ is bijective as a function from $[0,\infty)$ to $[0,\infty)$, a unique solution $\tau_0$ exists for every $\lambda$. 
It clearly holds that $\tau_0\to\infty$ as $\lambda\to\infty$ and thus $\tau_0\le\log\lambda$ for sufficiently large $\lambda$. We can conclude that
\begin{align*}
\limsup_{\lambda\to\infty}\frac{N(\lambda)}{\log^{2}\lambda}\le\frac{1}{(\pi b)^2}\,.
\end{align*}
To find an analogous lower bound we note that again by \eqref{eq:Hasymp} for $\lambda\ge2$
\begin{align*}
N(\lambda)\ge\sum_{j\ge1}\left(1-\frac{\lambda_j}{\lambda}\right)_+
=\frac{1}{\lambda}\sum_{j\ge1}(\lambda-\lambda_j)_+
\ge\frac{\log^{2}\lambda}{(\pi b)^2}+C\log\lambda
\end{align*}
with some constant $C>0$.

\subsection{The General Case $\zeta>0$}\label{subsec:Hgen}
It is straightforward to generalize the proof of Theorem \ref{th:Hasymp} to any $\zeta>0$. The coherent state representation of $W(\zeta)=V+\zeta V^{-1}$ can be computed to be
\begin{align*}
\sclp{W(\zeta)\psi,\psi}=d_2\iint_{\R^{2}}(\e^{2\pi by}+\zeta\e^{-2\pi by})|\widetilde{\psi}(k,y)|^2\dd k\dd y
\end{align*}
for $\psi\in\dom(W(\zeta))$. Repeating calculations of Sect.~\ref{subsec:upper} leads to an upper bound of the Riesz $\sum_{j\ge1}(\lambda-\lambda_j)_+$, which can be written as a sum of four integrals of the form $\int_0^\infty\int_0^\infty (\lambda-c_1\e^{2\pi bk}-c_2\e^{2\pi by})_+\dd k\dd y$. The asymptotic behaviour of these integrals was discussed in Sect.~\ref{subsec:upper}. A lower bound of the Riesz mean can be established by repeating verbatim the computations in Sect.~\ref{subsec:lower}, which proves Theorem \ref{th:Hasymp} and Corollary \ref{cor:Hasymp} for $\zeta>0$.

\section{The Operator $H_{m,n}$}
 
The operator $H_{m,n}=U+V+q^{-mn}U^{-m}V^{-n}$ is given by the following formal functional-difference expression
\begin{align*}
(H_{m,n}\psi)(x)=\psi(x+\ii b)+\e^{2\pi bx}\psi(x)+q^{-mn}\e^{-2\pi nbx}\psi(x-m\ii b)\,.
\end{align*} 
The operator $H_{m,n}$ is symmetric and non-negative on the domain $\psi\in\mathscr{D}$ consisting of linear combinations of the functions $p(x)\e^{-x^{2}+cx}$, where $p(x)$ is a polynomial and $c\in\C$.  Indeed, for $\psi\in\mathscr{D}$ it follows from the Weyl relation
$$U^{-m}\widetilde{V}^{-n}=q^{mn}\widetilde{V}^{-n}U^{-m},\quad\text{where}\quad \widetilde{V}=V^{1/2}=\e^{\pi bQ},$$
that 
\begin{equation}\label{Weyl-m-n}
q^{-mn}\langle U^{-m}V^{-n}\psi,\psi\rangle=\langle U^{-m}\widetilde{V}^{-n}\psi,\widetilde{V}^{-n}\psi\rangle\ge 0.
\end{equation}
Whence $H_{m,n}$ admits a Friedrichs extension and it what follows we will continue to denote it by $H_{m,n}$.
The spectrum of this operator consists of positive eigenvalues $\lambda_j$ that converge to infinity, $\lim_{j\to\infty}\lambda_j=\infty$. The proof of this statement is deferred to the end of Sect.~\ref{sec:coHmn} since it makes use of the coherent state representation of $H_{m,n}$. 

\begin{theorem}\label{th:Gasymp}
For $m,n\in\N$ the eigenvalues $\lambda_j$ of the operator $H_{m,n}$ have the following asymptotic behaviour
\begin{align*}
\sum_{j\ge1}(\lambda-\lambda_j)_+=\frac{c_{m,m}}{(2\pi b)^2}\lambda\log^{2}\lambda+O(\lambda\log\lambda)\quad\text{as}\quad\lambda\to\infty,
\end{align*}
where $c_{m,n}=\dfrac{(m+n+1)^{2}}{2mn}$.
\end{theorem}

Having established Theorem \ref{th:Gasymp}, the exact same argument as in Sect.~\ref{subsec:N} proves the following corollary.

\begin{corollary}\label{cor:Gasymp}
The number $N(\lambda)=\#\set{j\in\N}{\lambda_j<\lambda}$ of eigenvalues of $H_{m,n}$ less than $\lambda$ satisfies
\begin{align*}
\lim_{\lambda\to\infty}\frac{N(\lambda)}{\log^{2}\lambda}=\frac{c_{m,n}}{(2\pi b)^2}\,.
\end{align*}
\end{corollary}

In particular, this implies that the operator $H_{m,n}^{-1}$ is of trace class since
$$\sum_{j=1}^{\infty}\frac{1}{\lambda_{j}}=\int_{\lambda_{1}}^{\infty}\frac{1}{\lambda}\dd N(\lambda)=\left. \frac{N(\lambda)}{\lambda}\right|_{\lambda_{1}}^{\infty}+\int_{\lambda_1}^{\infty}\frac{N(\lambda)}{\lambda^{2}}\dd \lambda<\infty.$$

\begin{remark} As in Remark \ref{Weyl law}, the term $c_{m,n}\log^{2}\lambda/(2\pi b)^{2}$ is precisely the leading term as $\lambda\to\infty$ of the phase volume of the classical region $\{(k,x)\in\R^{2} : \e^{-2\pi b k}+\e^{2\pi b x}+\e^{2\pi b(mk-nx)}\le \lambda\}$. Similarly, the term $c_{m,n}\lambda\log^{2}\lambda/(2\pi b)^{2}$ coincides with the leading term in the phase space integral 
$$\iint_{\R^{2}}(\lambda-\e^{-2\pi b k}-\e^{2\pi b x}-\e^{2\pi b(mk-nx)})_+\dd k\dd x\quad\text{as}\quad\lambda\to\infty.$$
\end{remark}

As in Sect.~\ref{sec:H}, we first obtain a representation of $H_{m,n}$ using the coherent state transform and then prove the upper and lower bounds. The computations will closely follow those in Sect.~\ref{subsec:coherent}, \ref{subsec:upper} and \ref{subsec:lower}, and we will just highlight the main points.

\subsection{The Coherent State Representation}\label{sec:coHmn}
Let $\widetilde{\psi}$ again denote the coherent state transform of a function $\psi\in L^2(\R)$ with respect to the Gaussian function $g$. 
In complete analogy with Sect.~\ref{subsec:coherent}, identity \eqref{eq:Itrans2}, together with the facts that $U=\mathcal{F}^{-1}V^{-1}\mathcal{F}$ and $\e^{-2\pi bk}=\e^{-2\pi b(k-q)}\e^{2\pi bq}$, leads to the representation
\begin{align*}
\sclp[2]{U\psi,\psi}=d_1\iint_{\R^{2}} \e^{-2\pi bk}|\widetilde{\psi}(k,y)|^2\dd k\dd y\,.
\end{align*}
Here, we have used the symmetries of the functions involved to conclude that
\begin{align*}
\frac{1}{\sclp[2]{V^{-1}\widehat{g},\widehat{g}}}=\frac{2}{\sclp[2]{W\widehat{g},\widehat{g}}}=d_1\,.
\end{align*}
Similarly, 
\begin{align*}
\sclp[2]{U^{-m}\psi,\psi}=d^{m^{2}}_1\iint_{\R^{2}} \e^{2\pi bmk}|\widetilde{\psi}(k,y)|^2\dd k\dd y\,. 
\end{align*}
In the same way identity \eqref{eq:Itrans1} yields the representation
\begin{align*}
\sclp[2]{V\psi,\psi}=d_2\iint_{\R^{2}} \e^{2\pi by}|\widetilde{\psi}(k,y)|^2\dd k\dd y\,,
\end{align*}
where we have used that 
\begin{align*}
\frac{1}{\sclp[2]{Vg,g}}=\frac{1}{\sclp[2]{V^{-1}g,g}}=\frac{2}{\sclp[2]{Wg,g}}=d_2\,,
\end{align*}
since $g$ is even. 

To derive of the representation of the mixed term $q^{-mn}U^{-m}V^{-n}$ we use \eqref{Weyl-m-n} to get
$$q^{-mn}\sclp[2]{U^{-m}V^{-n}\psi,\psi}=\sclp[2]{U^{-m}\psi_{1},\psi_{1}}=d^{m^{2}}_1\iint_{\R^{2}} \e^{2\pi bmk}|\widetilde{\psi}_{1}(k,y)|^2\dd k\dd y,$$
where $\widetilde{\psi}_{1}(k,y)$ is the coherent state transform of the function $\psi_{1}(x)=(\tilde{V}^{-n}\psi)(x)=\e^{-\pi bnx}\psi(x)$. Completing the square, we obtain 
\begin{align*}
\widetilde{\psi}_{1}(k,y)& =\int_{\R}\e^{-2\pi \ii kx}g(x-y)\e^{-\pi bnx}\psi(x)\dd x=
\e^{\frac{(\pi nb)^{2}}{2a}-\pi bny }\,\tilde{\psi}\left(k,y-\tfrac{\pi nb}{a}\right),
\end{align*} 
so that
\begin{align*}
q^{-mn}\sclp[2]{U^{-m}V^{-n}\psi,\psi} &=d^{m^{2}}_1\e^{\frac{(\pi nb)^{2}}{a}}\iint_{\R^{2}} \e^{2\pi b(mk-ny)}|\widetilde{\psi}\left(k,y-\tfrac{\pi nb}{a}\right)|^2\dd k\dd y \\
 &=d^{m^{2}}_1d^{n^{2}}_{2}\iint_{\R^{2}} \e^{2\pi b(mk-ny)}|\widetilde{\psi}(k,y)|^2\dd k\dd y. 
\end{align*}
Summarizing, we obtain  the coherent state representation of the operator $H_{m,n}$,
\begin{gather} 
\sclp{H_{m,n} \psi,\psi} =\nonumber \\  
\iint_{\R^{2}} \left(d_1 \e^{-2\pi bk} + d_2 \e^{2\pi by} + d_3 \e^{2\pi b (m k-ny)}  
\right) |\widetilde\psi (k,y)|^2 \dd k \dd y, \label{H-mn-coh}
\end{gather}
where we put $d_{3}=d^{m^{2}}_1d^{n^{2}}_{2}$.

Using representation \eqref{H-mn-coh},
we can now prove that the spectrum of $H_{m,n}$ is discrete. 

\begin{proposition}\label{prop:Hmnspec}
The operator $H_{m,n}$ satisfies $H_{m,n}> cI$, where the constant $c>0$ depends on $m,n\in\N$, and has purely  discrete spectrum consisting of finite multiplicity positive eigenvalues tending to infinity.
\end{proposition}

\begin{proof}
According to \eqref{H-mn-coh}, the quadratic form of the operator $H_{m,n}$ is
\begin{align*}
\sclp{H_{m,n} \psi,\psi} =\iint_{\R^{2}} \Psi(k,y) |\widetilde\psi (k,y)|^2 \dd k \dd y,
\end{align*}
where 
\begin{align}\label{Psi}
\Psi(k,y) = d_1 \e^{-2\pi bk} + d_2 \e^{2\pi by} + d_3 \e^{2\pi b m k}  \e^{- 2\pi bn y}.
\end{align}
If $k\le 0$ and $y\ge0$, then omitting the last term in \eqref{Psi} we obtain
\begin{align*}
\Psi(k,y) \ge \frac{d_1}{2} (\e^{-\pi bk} + \e^{\pi bk} ) + \frac{d_2}{2} (\e^{\pi by} + \e^{-\pi by}) . 
\end{align*}
If $k\ge0$ and $y\le0$, then 
\begin{align*}
\e^{2\pi b m k}  \e^{- 2\pi bn y} \ge \frac12 \left( \e^{2\pi b m k}  +  \e^{- 2\pi bn y}\right)
\end{align*}
and therefore 
\begin{align*}
\Psi(k,y) \ge  d_1 \e^{-2\pi bk} + d_2 \e^{2\pi by} + \frac{d_3}{2} \left( \e^{2\pi b m k}  +  \e^{- 2\pi bn y}\right).
\end{align*}
Consider the case $k\ge 0$, $y\ge 0$. Assume that $\beta mk\ge ny$, where $\beta<1$. Then
\begin{align*}
\e^{2\pi b m k}  \e^{- 2\pi bn y} \ge \e^{2\pi b m (1-\beta) k}.
\end{align*}
If now $k\ge 0$, $y\ge 0$ and $\beta mk \le  ny$, then we omit the last term in \eqref{Psi} and use 
\begin{align*}
\e^{2\pi b y}\ge \frac12\left(\e^{2\pi b y}+\e^{\beta m n^{-1} k}\right).
\end{align*}
Similarly we treat the case $k\le 0$, $y\le0$.
Finally we conclude that there are positive constants $c_{1}$ and $c_{2}$ such that 
\begin{align}\label{ineq}
\Psi(k,y) > \Phi(k,y) := c_{1}\left(\e^{-c_{2}k} + \e^{c_{2}k}  + \e^{-c_{2}y}  + \e^{c_{2}y}\right).
\end{align}
Denote by $A$ the operator defined by the quadratic form 
\begin{align*}
\sclp{A\psi, \psi} := \iint_{\R^2} \Phi(k,y) |\widetilde\psi(k,y)|^2\dd k\dd y.
\end{align*}
Then \eqref{ineq} implies $H_{m,n} > A$ and it follows from the Plancherel theorem that $A\geq cI$, where $c=4c_{1}$. Obviously due to Proposition \ref{prop:Hspec}  the spectrum of $A$ is discrete. By the min-max principle we can conclude that the same holds for the spectrum of $H_{m,n}$ and the proof is complete. 
\end{proof}

\subsection{Deriving an Upper Bound}\label{subsec:upperG} Repeating the computation in Sect.~\ref{subsec:upper} and using Jensen's inequality we obtain
\begin{align*}
\sum_{j\ge1}(\lambda-\lambda_j)_+\le\iint_{\R^{2}}\left(\lambda-d_1\e^{-2\pi bk}-d_2\e^{2\pi b y}-d_3\e^{2\pi b(mk-ny)}\right)_+\dd k\dd y\,.
\end{align*}
To find an upper bound on the right-hand side, we separately consider all four quadrants of $\R^{2}$. 

If $k\le0, y\ge0$, an upper bound is obtained by omitting the mixed term $d_3\e^{2\pi b(mk-ny)}$. The double integral is then of the same form as the upper bound in Sect.~\ref{subsec:upper} and its leading term as $\lambda\to\infty$ is $\lambda\log^{2}\lambda/(2\pi b)^{2}$. 

If $k\ge0, y\le0$, we omit two exponentially decaying terms $d_1\e^{-2\pi b k}$ and $d_2\e^{2\pi b y}$. Changing variables $u_1=d_3\e^{2\pi bmk}$ and $u_2=\e^{-2\pi b ny}$, we obtain the double integral
\begin{align}
\frac{1}{mn(2\pi b)^{2}}\int_{d_3}^{\lambda}\int_{1}^{\lambda/u_1}\frac{\lambda-u_1u_2}{u_1 u_2}\dd u_2\dd u_1=\frac{\lambda\log^{2}\lambda}{2mn(2\pi b)^2} +O(\lambda\log\lambda)
\label{eq:int2} 
\end{align} 
as $\lambda\rightarrow\infty$, which can be easily verified by direct computation. 

In case $k\ge0,y\ge0$ we omit the term $d_1\e^{-2\pi b k}$ and changing variables $u_1=d_3\e^{2\pi b(mk-ny)}$ and $u_2=d_2\e^{2\pi by}$ yields the integral
\begin{align}
\frac{1}{m(2\pi b)^2}\int_{d_2}^{\tilde\lambda}\int_{d_4/u^{n}_2}^{\lambda-u_2}\frac{\lambda-u_1-u_2}{u_1u_2}\dd u_1\dd u_2\,,
\label{eq:int3} 
\end{align}
where $d_{4}=d_{3}d_{2}^{n}$ and $\tilde\lambda$ is the root of equation $\lambda=d_{4}u_{2}^{-n}+u_{2}$.
It is easy to see that for $\lambda\rightarrow\infty$ one can replace $\tilde\lambda$ by $\lambda$ and obtain the leading term $(n+2)\lambda\log^{2}\lambda/2m(2\pi b)^{2}$. The case  $k\le0,y\le0$ is treated similarly with the leading term $(m+2)\lambda\log^{2}\lambda/2n(2\pi b)^{2}$.

Summarizing, we arrive at the estimate
\begin{align*}
\sum_{j\ge1}(\lambda-\lambda_j)_+\le \frac{c_{m,n}}{(2\pi b)^2}\lambda\log^{2}\lambda+O(\lambda\log\lambda)\,.
\end{align*}

\subsection{Deriving a Lower Bound}\label{subsec:lowerG}
To derive a lower bound we repeat computations in Sect.~\ref{subsec:lower}. Denoting the projection-valued measure of $H_{m,n}$ on $[0,\infty)$ by $\mathrm{d}F_{\mu}$, we obtain upon an application of Jensen's inequality that
\begin{align*}
\sum_{j\ge1}(\lambda-\lambda_j)_+ 
\ge \iint_{\R^{2}}\left(\lambda-\int_0^\infty \mu\sclp[2]{\mathrm{d} F(\mu) e_{k,y},e_{k,y}}\right)_+\dd k\dd y
\end{align*}
for $\lambda\ge 0$.  The spectral theorem implies that
\begin{gather*}
\int_0^\infty \mu\sclp[2]{\mathrm{d} F(\mu) e_{k,y},e_{k,y}}
=\sclp[2]{H_{m,n}e_{k,y},e_{k,y}}\\
=\sclp[2]{Ue_{k,y},e_{k,y}}+\sclp[2]{Ve_{k,y},e_{k,y}}+q^{-mn}\sclp[2]{U^{-m}V^{-n}e_{k,y},e_{k,y}}\,.
\end{gather*}
The three inner products on the right-hand side can be computed explicitly and we get the inequality
\begin{align*}
\sum_{j\ge1}(\lambda-\lambda_j)_+ 
\ge \iint_{\R^{2}}\left(\lambda-\frac{1}{d_1}\e^{-2\pi b k}-\frac{1}{d_2}\e^{2\pi b y}-\frac{1}{d_3}\e^{2\pi b(mk-ny)}\right)_+\dd k\dd y\,.
\end{align*} 
To obtain a lower bound on the right-hand side, we again consider  separately all four quadrants of $\R^{2}$. 

If $k\le0, y\ge0$ we make the integrand smaller by replacing $\e^{2\pi b(mk-ny)}$ with $\e^{2\pi b y}$. The resulting double integral is of the form discussed in Sect.~\ref{subsec:upper} and its leading term as $\lambda\to\infty$ is $\lambda\log^{2}\lambda^2/(2 \pi b)^{2}$. 

In case $k\ge0,y\le0$, we decrease the right-hand side by replacing both $\e^{-2\pi b k}$ and $\e^{2 \pi b y}$ with $\e^{2\pi b (mk-ny)}$. This yields a double integral of the same form as \eqref{eq:int2} with the leading term $\lambda\log^{2}\lambda/2mn(2\pi b)^{2}$. 

For $k\ge0, y\ge0$ we bound $\e^{-2\pi b k}$ from above by $1$. The integral takes the same form as \eqref{eq:int3} with $\lambda$ replaced by $\lambda-1/d_1$. This does not affect the asymptotical behaviour $(n+2)\lambda\log^{2} \lambda/2m(2\pi b)^{2}$ as $\lambda\to\infty$.

The last case, $k\le 0,y\le0$, yields the leading term $(m+2)\lambda\log^{2} \lambda/2n(2\pi b)^{2}$ and we conclude that
\begin{align*}
\sum_{j\ge1}(\lambda-\lambda_j)_+\ge \frac{c_{m,n}}{(2\pi b)^2}\lambda\log^{2}\lambda+O(\lambda\log\lambda)\,.
\end{align*}
The proof of Theorem \ref{th:Gasymp} is complete.
\bibliographystyle{amsplain}
\bibliography{biblio_weyl}
\end{document}